\documentclass[reqno]{amsart}

\usepackage{amssymb}
\usepackage{graphicx}
\usepackage{amscd}
\usepackage[pagebackref]{hyperref}
\usepackage{color}
\usepackage{tabularx}
\usepackage[table]{xcolor}
\usepackage{float}
\usepackage{graphics,amsmath,amssymb}
\usepackage{amsthm}
\usepackage{amsfonts}
\usepackage{latexsym}
\usepackage{epsf}
\usepackage{xifthen}
\usepackage{mathrsfs}
\usepackage{dsfont}
\usepackage{makecell}
\usepackage{subfig}
\usepackage{amsmath}
\allowdisplaybreaks[4]
\usepackage{listings}
\usepackage{etoolbox}
\usepackage{fancyhdr}
\usepackage{pdflscape}
\usepackage[title,toc,titletoc]{appendix}
\usepackage{enumitem}
\usepackage[noadjust]{cite}
\usepackage{tikz}
\usetikzlibrary{automata,positioning,arrows}
\usepackage{young}
\usepackage[object=vectorian]{pgfornament} 
\usepackage{lipsum,tikz}
\usepackage{multirow}
\usepackage[OT2,T1]{fontenc}
\usepackage{mathtools}
\usepackage{ytableau}

\hypersetup{
	colorlinks=true, 
	linktoc=all, 
	linkcolor=blue} 

\numberwithin{equation}{section}

\theoremstyle{theorem}
\newtheorem{theorem}{Theorem}[section]
\newtheorem*{theorem*}{Theorem}

\newtheorem{lemma}[theorem]{Lemma}

\newtheorem{conjecture}[theorem]{Conjecture}

\newtheorem{innercustomgeneric}{\customgenericname}
\providecommand{\customgenericname}{}
\newcommand{\newcustomtheorem}[2]{%
	\newenvironment{#1}[1]
	{%
		\renewcommand\customgenericname{#2}%
		\renewcommand\theinnercustomgeneric{##1}%
		\innercustomgeneric
	}
	{\endinnercustomgeneric}
}
\newcustomtheorem{ctheorem}{Theorem}
\newcustomtheorem{clemma}{Lemma}

\theoremstyle{definition}

\newtheorem*{example*}{Example}
\newtheorem*{examples*}{Examples}
\newtheorem{remark}[theorem]{Remark}
\newtheorem*{remark*}{Remark}
\newtheorem*{remarks*}{Remarks}
\newtheorem*{note*}{Note}

\newtheoremstyle{named}{}{}{\itshape}{}{\bfseries}{.}{.5em}{\thmnote{#3} #1}
\theoremstyle{named}

\newtheoremstyle{customized}{}{}{\itshape}{}{\bfseries}{.}{.5em}{\thmnote{#3}}
\theoremstyle{customized}

\DeclareMathAlphabet{\mydutchcal}{U}{dutchcal}{m}{n}
\newcommand{\dcA}{\mydutchcal{A}}
\newcommand{\dcT}{\mydutchcal{T}}
\newcommand{\dcW}{\mydutchcal{W}}

\newcommand{\bfb}{\mathbf{b}}
\newcommand{\bfc}{\mathbf{c}}
\newcommand{\bfd}{\mathbf{d}}
\newcommand{\bfn}{\mathbf{n}}
\newcommand{\bfs}{\mathbf{s}}
\newcommand{\bfv}{\mathbf{v}}

\newcommand{\bfx}{\mathbf{x}}
\newcommand{\bfA}{\mathbf{A}}

\newcommand{\bfT}{\mathbf{T}}
\newcommand{\bfone}{\mathbf{1}}

\newcommand{\trans}{\mathsf{T}}

\newcommand{\dd}{\operatorname{d}}

\title[Tadpole Nahm sum]{Tadpole Nahm sum as a Wronskian}

\author[S. Chern]{Shane Chern}
\address[S. Chern]{Fakult\"at f\"ur Mathematik, Universit\"at Wien, Oskar-Morgenstern-Platz 1, Wien 1090, Austria}
\email{chenxiaohang92@gmail.com, xiaohangc92@univie.ac.at}

\author[C. Tran]{Chanh Tran}
\address[C. Tran]{MathVision AI}
\email{chanh@mathvision.ai}

\author[T. Wakhare]{Tanay Wakhare}
\address[T. Wakhare]{{}\textsuperscript{(1)}Department of Electrical Engineering and Computer Science, MIT, Cambridge, MA 02139, USA \newline \indent{}\textsuperscript{(2)}MathVision AI}
\email{tanay@mathvision.ai, twakhare@mit.edu}

\thanks{Mathematical contributors: S. Chern, T. Wakhare; Engineering contributor: C. Tran.}

\date{}

\keywords{Tadpole Nahm sum, Wronskian, generalized theta series, Macdonald-type identity, modularity.}

\subjclass[2020]{11P84, 05E10, 33D15.}

\begin{document}
	
\sloppy

\begin{abstract}
	We show that the tadpole Nahm sum is essentially a specialization of a series studied by Bartlett and Warnaar. This overlooked connection was first discovered by the internal model \texttt{mathvision-harness-0.3} at MathVision AI. We then use Macdonald-type identities for the Bartlett--Warnaar series to express the principal tadpole Nahm sum and its twisted version in terms of a Wronskian of generalized theta series. The principal case confirms a conjecture of Milas and Wang.
\end{abstract}

\maketitle

\section{Introduction}

A \emph{Nahm sum of rank $r$} is a $q$-summation taking the form
\begin{align*}
	\sum_{\bfn_r = (n_1,\ldots,n_r) \in \mathbb{Z}^r} \frac{q^{\frac{1}{2} \bfn_r \mathbf{A} \bfn_r^\trans + \mathbf{B} \bfn_r^\trans + \mathbf{C}}}{(q)_{n_1}\cdots (q)_{n_r}},
\end{align*}
where $\mathbf{A}$ is an $r$-dimensional rational positive definite square matrix, $\mathbf{B}$ is an $r$-dimensional rational vector, $\mathbf{C}$ is a rational scalar, and the \emph{$q$-Pochhammer symbol} is defined for $n \in \mathbb{N} \cup \{\infty\}$:
\begin{align*}
	(a)_n = (a;q)_n := \prod_{k=0}^{n-1} (1-a q^k).
\end{align*}
A typical example arises from the \emph{Rogers--Ramanujan identity}~\cite{Rog1894,RR1919}:
\begin{align*}
	\sum_{n\ge 0} \frac{q^{n^2}}{(q;q)_n} = \frac{1}{(q;q^5)_\infty (q^4;q^5)_\infty}.
\end{align*}
In particular, the modularity of the Rogers--Ramanujan sum on the left-hand side is evident from the theta product on the right. In the study of Nahm sums, one of the most central questions, now referred to as \emph{Nahm's problem}~\cite{Nah1994,Nah1995}, is to decide their \emph{modularity}. Zagier~\cite{Zag2007} was the first to make major progress for lower rank cases --- he proved that there are only seven rank one modular Nahm sums, and further offered lists of conjectural candidates for rank two~\cite[p.~47, Table~2]{Zag2007} and three~\cite[p.~49, Table~3]{Zag2007} cases, which were later resolved in \cite{CF2013,CRW2024,Wan2024a,Wan2024b,VZ2011}.

Starting from the groundbreaking work of Lepowsky and Wilson~\cite{LW1984,LW1985}, connections between Nahm sums and representations of infinite-dimensional Lie algebras have been receiving widespread attention. With this setup, the Nahm sums are typically expected to exhibit modularity in accordance with the Kac--Peterson theory~\cite{Kac1984}. In 2016, Calinescu, Milas, and Penn~\cite{CMP2016} considered the principal subspace of the basic $A_{2r}^{(2)}$-module and showed that its full character is a certain Nahm sum related to the \emph{tadpole Dynkin diagram} $T_r := A_{2r}/\mathbb{Z}_2$ with $A_{2r}$ the type $A$ root lattice.

Let $\bfT_r$ be the \emph{Cartan matrix} of $T_r$ given by
\begin{align*}
	\bfT_r(i,j)_{1\le i,j\le r} := \begin{cases}
		1, & \text{if $i = j = r$},\\
		2, & \text{if $i = j < r$},\\
		-1, & \text{if $i = j \pm 1$},\\
		0, & \text{otherwise}.
	\end{cases}
\end{align*}
The \emph{tadpole Nahm sum of rank $r$} constructed in \cite[p.~1781, eq.~(5.19)]{CMP2016} is
\begin{align}
	\dcT_r(\bfx_r;q) = \dcT_r(x_1,\ldots,x_r;q) := \sum_{\bfn_r = (n_1,\ldots,n_r)\in \mathbb{Z}^r} \frac{q^{\frac{1}{2} \bfn_r \bfT_r \bfn_r^\trans} x_1^{n_1}\cdots x_r^{n_r}}{(q)_{n_1} \cdots (q)_{n_r}}.
\end{align}
We say the tadpole Nahm sum is \emph{principal} when we take $x_1=\cdots=x_r=1$.

Toward modularity, the rank two principal sum $\dcT_2(1,1;q)$ was confirmed in the affirmative by Calinescu, Milas, and Penn themselves in \cite[p.~1782, eq.~(6.1)]{CMP2016}, and subsequently similar results were established by Milas and Wang for the rank three sum~\cite[p.~77, eq.~(1.11)]{MW2024}, and by Shi and Wang in ranks four~\cite[p.~218, eq.~(1.20)]{SW2016} and five~\cite[p.~219, eq.~(1.25)]{SW2016}. In general, Calinescu, Milas, and Penn~\cite[p.~1781, Conjecture~6.1]{CMP2016} made the following modularity conjecture.

\begin{conjecture}[Calinescu--Milas--Penn]\label{conj:CMP}
	The principal tadpole Nahm sum $\dcT_r(1,\ldots,1;q)$, shifted by a rational power of $q$, is modular.
\end{conjecture}

To give a complete resolution of the Calinescu--Milas--Penn conjecture, Milas and Wang~\cite[p.~100, Conjecture~4.5]{MW2024} predicted an explicit expression for the principal tadpole Nahm sum, the equivalent form of which will be stated in the $\delta=0$ case of Theorem~\ref{th:main} below.

\begin{conjecture}[Milas--Wang]
	The principal tadpole Nahm sum $\dcT_r(1,\ldots,1;q)$, shifted by a rational power of $q$, can be explicitly expressed as the product of a modular infinite product and the Wronskian of a length-$r$ sequence of generalized theta series.
\end{conjecture}

\begin{remark}
	The modularity of such a Wronskian is a standard result; see the discussions in \cite[Section~4.2]{MW2024}. Therefore, Milas and Wang's explicit expression implies the Calinescu--Milas--Penn conjecture.
\end{remark}

Given functions $g_1(q),\ldots,g_r(q)$, their \emph{Wronskian} with respect to the operator $D:= q \frac{\dd}{\dd q}$ is defined by
\begin{align*}
	\dcW_D(g_1,\ldots,g_r) := \det\big(D^{i-1}g_j\big)_{1\le i,j\le r},
\end{align*}
where $D^k g$ means applying the operator $D$ to $g$ for $k$ iterations. For a formal series
\begin{align*}
	g(q) = \sum_{\substack{v\in \mathbb{Q}\\ v\ge V}} c_v q^v,
\end{align*}
with $V\in \mathbb{Q}$ and $c_V\ne 0$, define its \emph{normalization} by
\begin{align*}
	\widetilde{g}(q) := \sum_{\substack{v\in \mathbb{Q}\\ v\ge V}} \frac{c_v}{c_V} q^{v}.
\end{align*}
We denote by $\widetilde{\dcW}_D(g_1,\ldots,g_r)$ the normalization of $\dcW_D(g_1,\ldots,g_r)$.

Now for $m\in \mathbb{Z}_{>0}$ and $s\in \mathbb{Q}$ such that $|s| < \frac{m}{2}$, we introduce the \emph{generalized theta series}
\begin{align*}
	\Theta_{s,m}^{k,\epsilon}(q) := \sum_{\substack{v\in \mathbb{Q}\\ v\equiv s \bmod{m}}} \epsilon^{\frac{v-s}{m}} v^k q^{\frac{v^2}{2m}},
\end{align*}
where $k\in \mathbb{Z}$ and $\epsilon\in \{\pm 1\}$.

In this note, we not only prove the explicit expression for the principal sum $\dcT_r(1,\ldots,1,1;q)$ conjectured in \cite[p.~100, Conjecture~4.5]{MW2024}, but also consider the twisted version $\dcT_r(1,\ldots,1,q^{\frac{1}{2}};q)$. Our main result is stated as follows.

\begin{theorem}\label{th:main}
	For $\delta\in \{0,\frac{1}{2}\}$,
	\begin{align}
		\dcT_{r}(\underbrace{1,\ldots,1}_{\text{$r-1$ terms}},q^{\delta};q) = q^{-\frac{1}{8(r+2)}\binom{r+1-2\delta}{3}} \frac{(-q^{\frac{1}{2}+\delta})_\infty^{r}}{(q)_\infty^{\binom{r}{2}}} \,\widetilde{\dcW}_D\big(\Theta_{\frac{r+1}{2}-\delta-i,r+2}^{r-2\lfloor\frac{r}{2}\rfloor,(-1)^{r+1}}\big)_{1\le i\le \lfloor\frac{r}{2}\rfloor}.
	\end{align}
	Consequently, both $\dcT_{r}(1,\ldots,1,1;q)$ and $\dcT_{r}(1,\ldots,1,q^{\frac{1}{2}};q)$, shifted by a certain rational power of $q$, are modular.
\end{theorem}

Our key observation is that the tadpole Nahm sum $\dcT_r(\bfx_r;q)$ is essentially a specialization of a series studied by Bartlett and Warnaar~\cite[p.~1085, eq.~(3.21)]{BW2015}. Moreover, Bartlett and Warnaar~\cite[Section~6]{BW2015} established a collection of Macdonald-type identities for their series. Thus, it only remains to rewrite these identities in terms of the desired Wronskians, as will be demonstrated in the next section.

\section{Proof of the main result}

Let $\bfA_r$ be the \emph{type $A_r$ Cartan matrix} wherein
\begin{align*}
	\bfA_r(i,j)_{1\le i,j\le r} := \begin{cases}
		2, & \text{if $i = j$},\\
		-1, & \text{if $i = j \pm 1$},\\
		0, & \text{otherwise}.
	\end{cases}
\end{align*}
We introduce the sum
\begin{align}
	\dcA_r(\bfx_r,z;q) = \dcA_r(x_1,\ldots,x_r,z;q) := \sum_{\bfn_r\in \mathbb{Z}^r} \frac{q^{\frac{1}{2} \bfn_r \bfA_r \bfn_r^\trans} x_1^{n_1}\cdots x_r^{n_r} (-zq^{\frac{1}{2}-n_r})_{n_r}}{(q)_{n_1} \cdots (q)_{n_r}}.
\end{align}

\begin{lemma}\label{le:H-F}
	We have
	\begin{align}\label{eq:H-F}
		\dcT_{r+1}(\bfx_{r+1};q) = (-x_{r+1}q^{\frac{1}{2}})_\infty \dcA_r(\bfx_r,x_{r+1};q).
	\end{align}
\end{lemma}

\begin{proof}
	Note that
	\begin{align*}
		\dcT_{r+1}(\bfx_{r+1};q) = \sum_{\bfn_r \in \mathbb{Z}^r} \frac{q^{\frac{1}{2} \bfn_r \bfA_r \bfn_r^\trans} x_1^{n_1}\cdots x_r^{n_r}}{(q)_{n_1} \cdots (q)_{n_r}} \sum_{n_{r+1}\ge 0} \frac{q^{\frac{1}{2}n_{r+1}^2 - n_r n_{r+1}} x_{r+1}^{n_{r+1}}}{(q)_{n_{r+1}}}.
	\end{align*}
	For the inner sum over $n_{r+1}$, we apply \emph{Euler's second sum} \cite[p.~354, eq.~(II.2)]{GR2004}:
	\begin{align*}
		\sum_{n\ge 0} \frac{q^{\binom{n}{2}}z^n}{(q)_n} = (-z;q)_\infty.
	\end{align*}
	The desired relation then follows.
\end{proof}

Write
\begin{align*}
	\bfone_r := (\underbrace{1,1,\ldots,1}_{\text{$r$ terms}}).
\end{align*}
The series $\dcA_r(\bfone_r,z;q)$ is the $m=1$ case of \cite[p.~1085, eq.~(3.21)]{BW2015} with $(u,w)\mapsto (1,0)$. It is important to note that the $m=1$ case of this Bartlett--Warnaar series equals a certain sum of modified Hall--Littlewood polynomials, as proven in \cite[p.~1086, Theorem~3.7]{BW2015}. As such, we may further take advantage of the Macdonald-type identities provided in \cite[Section~6]{BW2015}, originating from Macdonald's seminal work~\cite{Mac1972} on the Dedekind eta function.

For $\bfv_r = (v_1,\ldots,v_r)$ and $k\in \mathbb{Z}_{\ge 0}$, let
\begin{align*}
	|\bfv_r| := v_1 + \cdots + v_r,\qquad\qquad \lVert \bfv_r\rVert^2 := v_1^2 + \cdots + v_r^2.
\end{align*}
In addition, write
\begin{align*}
	\chi_k(\bfv_r) &:= \left(\prod_{i=1}^r v_i\right)^k \left(\prod_{1\le i<j \le r} (v_i^2 - v_j^2)\right).
\end{align*}
Define the classical \emph{Weyl vectors}:
\begin{align*}
	\bfb_r &:= (r-\tfrac{1}{2},\ldots,\tfrac{3}{2},\tfrac{1}{2}),\\
	\bfc_r &:= (r,\ldots,2,1),\\
	\bfd_r &:= (r-1,\ldots,1,0).
\end{align*}
The following Macdonald-type identities for the Bartlett--Warnaar series at $m=1$ were presented in \cite[Section~6]{BW2015} in an equivalent form:\footnote{Using the notation of Bartlett and Warnaar, these series are respectively $F_{m,2r}(0,1;q)$, $F_{m,2r-1}(0,1;q)$, $F_{m,2r}(0,q^{\frac{1}{2}};q)$, and $F_{m,2r-1}(0,q^{\frac{1}{2}};q)$ at $m=1$ in \cite[Section~6]{BW2015}. In particular, the Macdonald-type identities in \cite[Section~6]{BW2015} hold for $m=1$ due to \cite[p.~1086, Theorem~3.7]{BW2015}.}
\begin{subequations}\label{eq:BW-sum}
	\begin{align}
		&\dcA_{2r}(\bfone_{2r},1;q)\notag\\
		&\qquad = q^{-\frac{1}{8(2r+3)}\binom{2r+2}{3}} \frac{(-q^{\frac{1}{2}})_\infty^{2r}}{(q)_\infty^{2r^2+r}} \sum_{\substack{\bfv_r \in \mathbb{Z}^r\\ \bfv_r \equiv \bfc_r \bmod 2r+3}} \frac{\chi_1(\bfv_r)}{\chi_1(\bfc_r)} q^{\frac{\lVert \bfv_r\rVert^2}{4r+6}},\\
		&\dcA_{2r-1}(\bfone_{2r-1},1;q)\notag\\
		&\qquad = q^{-\frac{1}{8(2r+2)}\binom{2r+1}{3}} \frac{(-q^{\frac{1}{2}})_\infty^{2r-1}}{(q)_\infty^{2r^2-r}} \sum_{\substack{\bfv_r \in \frac{1}{2}\mathbb{Z}^r\\ \bfv_r \equiv \bfb_r \bmod 2r+2}} (-1)^{\frac{|\bfv_r| - |\bfb_r|}{2r+2}}\frac{\chi_0(\bfv_r)}{\chi_0(\bfb_r)} q^{\frac{\lVert \bfv_r\rVert^2}{4r+4}},\\
		&\dcA_{2r}(\bfone_{2r},q^{\frac{1}{2}};q)\notag\\
		&\qquad = q^{-\frac{1}{8(2r+3)}\binom{2r+1}{3}} \frac{(-q)_\infty^{2r}}{(q)_\infty^{2r^2+r}} \sum_{\substack{\bfv_r \in \frac{1}{2}\mathbb{Z}^r\\ \bfv_r \equiv \bfb_r \bmod 2r+3}} \frac{\chi_1(\bfv_r)}{\chi_1(\bfb_r)} q^{\frac{\lVert \bfv_r\rVert^2}{4r+6}},\\
		&\dcA_{2r-1}(\bfone_{2r-1},q^{\frac{1}{2}};q)\notag\\
		&\qquad = q^{-\frac{1}{8(2r+2)}\binom{2r}{3}} \frac{(-q)_\infty^{2r-1}}{(q)_\infty^{2r^2-r}} \sum_{\substack{\bfv_r \in \mathbb{Z}^r\\ \bfv_r \equiv \bfd_r \bmod 2r+2}} (-1)^{\frac{|\bfv_r| - |\bfd_r|}{2r+2}}\frac{\chi_0(\bfv_r)}{\chi_0(\bfd_r)} q^{\frac{\lVert \bfv_r\rVert^2}{4r+4}}.
	\end{align}
\end{subequations}

\begin{lemma}
	Let $m\in \mathbb{Z}_{>0}$. Suppose $s_1,\ldots,s_r\in \mathbb{Q}$ are such that $|s_i| < \frac{m}{2}$ for every $s_i$. Then
	\begin{align}\label{eq:Wron-theta-normal}
		\widetilde{\dcW}_D(\Theta_{s_1,m}^{k,\epsilon}, \ldots, \Theta_{s_r,m}^{k,\epsilon}) = \sum_{\substack{\bfv_r\in \mathbb{Q}^r\\ \bfv_r \equiv \bfs_r \bmod{m}}} \epsilon^{\frac{|\bfv_r|-|\bfs_r|}{m}} \frac{\chi_k(\bfv_r)}{\chi_k(\bfs_r)} q^{\frac{\lVert \bfv_r\rVert^2}{2m}}.
	\end{align}
\end{lemma}

\begin{proof}
	Given a family of formal series $f_1(q),\ldots,f_r(q)$ where for every $1\le i\le r$,
	\begin{align*}
		f_i(q) = \sum_{n\in \mathbb{Z}} \alpha_i(n) q^{\kappa_i(n)},
	\end{align*}
	with $\alpha_i: \mathbb{Z} \to \mathbb{C}$ and $\kappa_i: \mathbb{Z} \to \mathbb{Q}$, it is clear that
	\begin{align*}
		\dcW_D(f_1,\ldots,f_r) &= \sum_{\bfn_r \in \mathbb{Z}^r} \det\big(\alpha_j(n_j) \kappa_j(n_j)^i q^{\kappa_j(n_j)}\big)_{1\le i,j\le r}\\
		&= \sum_{\bfn_r \in \mathbb{Z}^r} q^{\sum_{j=1}^r \kappa_j(n_j)} \left(\prod_{j=1}^r \alpha_j(n_j)\right) \det\big(\kappa_j(n_j)^i\big)_{1\le i,j\le r}.
	\end{align*}
	By evaluating the Vandermonde determinant~\cite[p.~5, eq.~(2.1)]{Kra1999}, we have
	\begin{align}\label{eq:Wronskian-general}
		\dcW_D(f_1,\ldots,f_r) = \sum_{\bfn_r \in \mathbb{Z}^r} \left(\prod_{j=1}^r \alpha_j(n_j)\right) \left(\prod_{1\le i<j \le r} \kappa_j(n_j) - \kappa_i(n_i)\right) q^{\sum_{j=1}^r \kappa_j(n_j)}.
	\end{align}
	In \eqref{eq:Wronskian-general}, we substitute our generalized theta series $\Theta_{s_1,m}^{k,\epsilon},\ldots,\Theta_{s_r,m}^{k,\epsilon}$ and derive
	\begin{align*}
		\dcW_D(\Theta_{s_1,m}^{k,\epsilon}, \ldots, \Theta_{s_r,m}^{k,\epsilon}) = \sum_{\substack{\bfv_r\in \mathbb{Q}^r\\ \bfv_r \equiv \bfs_r \bmod{m}}} \epsilon^{\frac{|\bfv_r|-|\bfs_r|}{m}} \big({-\tfrac{1}{2m}}\big)^{\binom{r}{2}} \chi_k(\bfv_r) q^{\frac{\lVert \bfv_r\rVert^2}{2m}}.
	\end{align*}
	Finally, in this Wronskian, the lowest power of $q$ takes place when $\bfv_r = \bfs_r$ because $|s_i| < \frac{m}{2}$ for every $s_i$. The desired normalization then follows.
\end{proof}

In light of \eqref{eq:Wron-theta-normal}, the four Macdonald-type identities in \eqref{eq:BW-sum} can be reformulated in a uniform way as
\begin{align}
	\dcA_{r-1}(\bfone_{r-1},q^{\delta};q) = q^{-\frac{1}{8(r+2)}\binom{r+1-2\delta}{3}} \frac{(-q^{\frac{1}{2}+\delta})_\infty^{r-1}}{(q)_\infty^{\binom{r}{2}}} \,\widetilde{\dcW}_D\big(\Theta_{\frac{r+1}{2}-\delta-i,r+2}^{r-2\lfloor\frac{r}{2}\rfloor,(-1)^{r+1}}\big)_{1\le i\le \lfloor\frac{r}{2}\rfloor},
\end{align}
where $\delta\in \{0,\frac{1}{2}\}$. Now invoking \eqref{eq:H-F} immediately confirms the claimed Theorem~\ref{th:main}.

\section{Remarks on AI use}

MathVision,\footnote{Available at \url{https://mathvision.ai}.} an interactive AI workspace for mathematical research, was used during the development of this work. We highlight that through a few general prompts for literature search, the internal model \texttt{mathvision-harness-0.3} identified the overlooked connection between the tadpole Nahm sum and the Bartlett--Warnaar series, as recorded in Lemma~\ref{le:H-F}. The MathVision workspace was then used to explore this connection by one of the mathematical contributors (T.W.). On the other hand, with knowledge of the aforementioned connection, the other mathematical contributor (S.C.) decided to directly move on to the paper by Bartlett and Warnaar~\cite{BW2015}. In doing so, the arguments after Lemma~\ref{le:H-F} in the present work were completed independently, without consulting the AI outputs.

\subsection*{Acknowledgements}

Shane Chern was supported by the FWF Austrian Science Fund (No.~10.55776/F1002).

\bibliographystyle{amsplain}

\end{document}